\documentclass[12pt]{amsart}
\usepackage{amscd,amsmath,amssymb,enumerate,amsfonts,mathrsfs,txfonts}
\usepackage{comment}
\usepackage{hyperref}
\usepackage{xcolor}
\usepackage{tikz,pgfplots,pgf}
\usepackage[all]{xy}
\usepackage{lipsum}

\tikzset{node distance=3cm, auto}
\usetikzlibrary{calc} 
\usetikzlibrary{trees} 

\usepackage{vmargin}
\setpapersize{A4}
\setmargins{2.5cm}      
{1.5cm}                        
{16.5cm}                      
{23.42cm}                   
{10pt}                          
{1cm}                           
{0pt}                          
{2cm}

\newtheorem{theorem}{Theorem}[section]
\newtheorem{proposition}[theorem]{Proposition}
\newtheorem{definition}[theorem]{Definition}
\newtheorem{corollary}[theorem]{Corollary}

\newtheorem{lemma}[theorem]{Lemma}

\def\A{\mathcal{A}}

\def\HL{\mathcal{H}L_0}
\def\Hl{\mathcal{H}l_0}
\def\G{\mathcal{G}_0}
\def\L{\mathcal{L}}
\def\B{\mathcal{B}}
\def\H{\mathcal{H}}

\def\F{\mathcal{F}}
\def\K{\mathcal{K}}

\def\QN{\mathcal{QN}}
\def\W{\mathcal{W}}
\def\N{\mathcal{N}}
\def\C{\mathbb{C}}

\def\R{\mathbb{R}}

\def\Im{\mathrm{Im}}
\def\abco{\mathrm{abco}}
\def\co{\mathrm{co}}
\def\lin{\mathrm{lin}}

\def\Lipo{\mathrm{Lip}_0}

\begin{document}

\title[$\mathcal{A}$-compact holomorphic Lipschitz mappings]{$\mathcal{A}$-compact holomorphic Lipschitz mappings\\ on the unit ball of a Banach space}

\author[A. Jim\'enez-Vargas]{A. Jim\'enez-Vargas}
\address[A. Jim\'enez-Vargas]{Departamento de Matem\'aticas, Universidad de Almer\'ia, Ctra. de Sacramento s/n, 04120 La Ca\~{n}ada de San Urbano, Almer\'ia, Spain}
\email{ajimenez@ual.es}

\author[D. Ruiz-Casternado]{D. Ruiz-Casternado}
\address[D. Ruiz-Casternado]{Departamento de Matem\'aticas, Universidad de Almer\'ia, Ctra. de Sacramento s/n, 04120 La Ca\~{n}ada de San Urbano, Almer\'ia, Spain}
\email{drc446@ual.es}

\date{\today}

\thanks{This research was partially supported by Junta de Andaluc\'ia grant FQM194. The first author was supported by Ministerio de Ciencia e Innovaci\'{o}n grant PID2021-122126NB-C31 (funded by MICIU/AEI/10.13039/501100011033 and by ERDF/EU); and the second author by an FPU predoctoral fellowship of the Spanish Ministry of Universities (FPU23/03235).}

\begin{abstract}
Let $X$ and $Y$ be complex Banach spaces, $B_X$ be the open unit ball of $X$ and $\HL(B_X,Y)$ be the Banach space of all holomorphic Lipschitz maps $f\colon B_X\to Y$ such that $f(0)=0$, endowed with the Lipschitz norm. Given a Banach operator ideal $\A$, we use the property of $\A$-compactness by Carl and Stephani to introduce and study the subclass of those functions in $\HL(B_X,Y)$ for which its Lipschitz image is a relatively $\A$-compact subset of $Y$. We focus our attention on its structure as a composition Banach holomorphic Lipschitz ideal by using its connection with $\A$-compact linear operators through linearization/transposition techniques. 
\end{abstract}
\maketitle


\section*{Introduction and preliminaries}

Let $X$ and $Y$ be complex Banach spaces and let $B_X$ be the open unit ball of $X$. If $\Lipo(B_X,Y)$ denotes the set of all zero-preserving Lipschitz maps from $B_X$ into $Y$ and $\H^\infty(B_X,Y)$ stands for the set of all bounded holomorphic maps from $B_X$ into $Y$, let 
$$
\HL(B_X,Y):=\Lipo(B_X,Y)\cap \H^\infty(B_X,Y)
$$
be the Banach space of all holomorphic Lipschitz maps $f\colon B_X\to Y$ such that $f(0)=0$, endowed with the Lipschitz norm
$$
L(f)=\sup\left\{\frac{\|f(x)-f(y)\|}{\|x-y\|}\colon x,y \in B_X,\, x\neq y\right\}.
$$
The little holomorphic Lipschitz space $\Hl(B_X,Y)$ is the closed subspace of $\HL(B_X,Y)$ consisting of all those mappings $f\colon B_X\to Y$ which satisfy the following property:
$$
\forall\varepsilon>0,\quad \exists\delta>0\colon x,y\in B_X,\, 0<\left\|x-y\right\|<\delta\quad \Rightarrow\quad \frac{\left\|f(x)-f(y)\right\|}{\left\|x-y\right\|}<\varepsilon.
$$
We will write $\HL(B_X)$ and $\Hl(B_X)$ instead of $\HL(B_X,\C)$ and $\Hl(B_X,\C)$, respectively.

Recently, the linearization of such holomorphic Lipschitz mappings has been addressed by Aron et al. in \cite{Aro24}. Using the Dixmier--Ng Theorem, they not only showed that $\HL(B_X,Y)$ is indeed a dual space, but that there exist a complex Banach space $\G(B_X)$ and a map $\delta_X\in\HL(B_X,\G(B_X))$ satisfying the following universal property: for every complex Banach space $Y$ and every map $f\in\HL(B_X,Y)$, there exists a unique operator $T_f\in\L(\G(B_X),Y)$ such that $T_f\circ\delta_X=f$. In fact, $\G(B_X)=\overline{\lin}(\delta_X(B_X))\subseteq\HL(B_X)^*$ and $\delta_X\colon B_X\to\G(B_X)$ is the isometry defined by $\delta_X(x)(g)=g(x)$ for all $x\in B_X$ and $g\in\HL(B_X)$. Furthermore, the map $f\mapsto T_f$ is an isometric isomorphism from $\HL(B_X,Y)$ onto $\L(\G(B_X),Y)$.  

Given a Banach operator ideal $\A$, our objective in this paper is to study the $\A$-compactness and the $\A$-boundedness of holomorphic Lipschitz maps from $B_X$ into $Y$. 

Carl and Stephani \cite{CarSte84} introduced a refinement of the notion of compactness involved to a given Banach operator ideal $[\mathcal{A},\left\|\cdot\right\|_\A]$. For a Banach space $X$, a set $K\subseteq X$ is said to be relatively $\A$-compact if there exist a Banach space $Y$, an operator $T\in\A(Y,X)$ and a relatively compact set $M\subseteq Y$ such that $K\subseteq T(M)$. We denote by $\mathcal{K}_\A(X)$ the collection of all relatively $\A$-compact subsets of $X$. Lassalle and Turco \cite{LasTur13} proposed a measure of the size of $\A$-compactness of $K\in\mathcal{K}_\A(X)$ defining $m_\A(K)=\inf\{\|T\|_\A\}$, where the infimum is taken over all such $Y$, $T$ and $M$ as above. 

The $\A$-compactness yields the so-called $\A$-compact operators that are those $T\in\L(X,Y)$ satisfying that $T(\overline{B}_X)$ is a relatively $\A$-compact subset of $Y$, where $\overline{B}_X$ denotes the closed unit ball of $X$. The set of such operators is denoted as $\K_\A(X,Y)$, and $[\K_\A,\|\cdot\|_{\K_\A}]$ is a Banach operator ideal with the norm
$$
\|T\|_{\K_\A} = m_\A(T(\overline{B}_X)) \qquad (T \in \K_\A(X,Y))
$$
(see \cite[Section 2]{LasTur13}). Some examples of operator ideals $\K_\A$ are particularly interesting. For $p \in [1,\infty)$, the ideal of $p$-compact operators, -- introduced by Sinha and Karn \cite{SinKar02} -- , coincides with the ideal of $\mathcal{N}_p$-compact operators, where $\mathcal{N}_p$ is the ideal of right $p$-nuclear operators (see \cite{Rya02}). Moreover, Ain, Lillemets and Oja \cite{AinLilOja12,AinOja15} generalized the ideal of $p$-compact operators introducing the ideal of $(p,q)$-compact operators for $p\in [1,\infty]$ and $q\in [1,p^*]$, which in this case coincides with the ideal of $\mathcal{N}_{(p,1,q^*)}$-compact operators, where $\mathcal{N}_{(p,1,q^*)}$ stands for the ideal of $(p,1,q^*)$-nuclear operators (see \cite[18.1.1]{Pie80}).

The Banach ideals of $\A$-compact operators, polynomials and holomorphic functions and their related approximation properties have been studied in \cite{Kim-19, LasTur13, LasTur18, Tur16}. In particular, the ideals of $p$-compact and $(p,q)$-compact operators have been extended to different non-linear contexts as, for example, holomorphic \cite{AroMaeRue10, LasTur12}, bounded-holomorphic \cite{Jim23,JimRui243}, polynomial \cite{AroRue11}, and Lipschitz \cite{AchDahTur19}.


We have divided this paper into three sections. Section \ref{Section 1} contains some preliminary results where we introduce the concept of ideal of holomorphic Lipschitz maps (Subsection \ref{Subsection 1}) and present both procedures to generate new Banach holomorphic Lipschitz ideals by composition (Subsection \ref{Subsection 2}) and duality (Subsection \ref{Subsection 3}). In particular, the transposition of functions permit us to identify in Corollary \ref{teo-4-1} the space $\HL(B_X,Y)$ with the space of all weak*-to-weak* continuous linear operators from $Y^*$ into $\HL(B_X)$. 

Given an operator ideal $\A$, our objective in Section \ref{Section 2} is to introduce a variant for holomorphic Lipschitz maps of the concept of $\A$-compact linear operators between Banach spaces \cite{CarSte84,LasTur13}. It is important to note that we do not use the global range of a map $f \in \HL(B_X,Y)$, but its Lipchitz image 
$$
\Im_{L}(f) := \left\{\frac{f(x)-f(y)}{\|x-y\|}\colon x,y\in B_X,\, x\neq y\right\} \subseteq Y.
$$
So, a map $f\in\HL(B_X,Y)$ is said to be $\A$-compact holomorphic Lipschitz if $\Im_L(f)$ is a relatively $\A$-compact subset of $Y$. In Propositions \ref{Theorem: characterizations A-compact} and \ref{Theorem: characterizations A-compact-2}, we obtain some immediate characterizations of $\A$-compact holomorphic Lipschitz maps applying the different descriptions of $\A$-compacts sets obtained in \cite{CarSte84,LasTur13}. Theorem \ref{Theorem: linearization A-compact} provides a useful characterization of such maps in terms of their linearizations. In particular, interesting descriptions of $\K$-compact and $\K_p$-compact holomorphic Lipschitz maps in term of the continuity and compactness of their transposes are stated (see Theorem \ref{teo-4-2}, and Propositions \ref{cor-4-1}-\ref{teo-4-22}). 

The $\A$-compactness is generalized by the property of $\A$-boundedness due to Stephani \cite{Ste-80}. An extension of this property to the holomorphic Lipschitz setting is presented in Section \ref{Section 3}. This approach permits us to study some new classes of holomorphic Lipschitz maps that are not covered by the $\A$-compactness such as weakly compact, Rosenthal, Banach--Saks, Asplund, finite-rank and approximable holomorphic Lipschitz maps. Some variants of classical results on ideals of factorizable operators are established for the associate ideals of holomorphic Lipschitz mappings (see Theorems \ref{teo-4-3}-\ref{teo-4-3-1}). By transposition, we identify certain spaces of holomorphic (little) Lipschitz maps from $B_X$ into $Y$ with some distinguished spaces of linear operators between  $Y^*$ and $\HL(B_X)$ equipped both with suitable topologies (see Propositions \ref{cor-4-1}-\ref{teo-4-22}-\ref{cor-4-1b} and Theorems \ref{teo-4-4}-\ref{Theorem: linearization finite-dimensional}). For future reference, we conclude the paper with an appeal to the approximation property of $\HL(B_X)$.


\section{Preliminary results}\label{Section 1}

\textbf{Notation.} Throughout this paper, $X$ and $Y$ will be complex Banach spaces. We denote by $\L(X,Y)$, $\F(X,Y)$, $\W(X,Y)$, $\K(X,Y)$ and $\overline{\F}(X,Y)$ the spaces of all bounded, finite-rank bounded, weakly compact, compact and approximable linear operators from $X$ to $Y$, respectively, endowed with the canonical operator norm $\left\|\cdot\right\|$. As usual, $X^*$ and $X^{**}$ stand for the topological dual and the topological bidual of $X$, respectively. We denote by $B_X$ and $\overline{B}_X$ the open unit ball and the closed unit ball of $X$, respectively. Given $A\subseteq X$, $\co(A)$, $\abco(A)$ and $\overline{\abco}(A)$ represent the convex hull, the absolutely convex hull and the norm-closed absolutely convex hull of $A$ in $X$. For any $x\in X$ and $x^*\in X^*$, we will use the notation $\langle x^*,x\rangle=x^*(x)$. We denote by $\HL(B_X,B_Y)$ the set of all zero-preserving holomorphic Lipschitz maps from $B_X$ into $B_Y$.

If $E$ and $F$ are locally convex Hausdorff spaces, $\L(E;F)$ stands for the linear space of all continuous linear operators from $E$ into $F$. Unless stated otherwise, if $E$ and $F$ are Banach spaces, we will understand that they are endowed with the norm topology. We refer to the monograph \cite{Meg-98} for definitions and main properties of the weak* topology $w^*$, the weak topology $w$ and the bounded weak* topology $bw^*$. The rest of the necessary notation will be introduced throughout the paper.


\subsection{Ideals of holomorphic Lipschitz mappings}\label{Subsection 1}

Based on the notion of operator ideal introduced by Pietsch \cite[1.1]{Pie80}, we present a version of this concept for holomorphic Lipschitz maps.

\begin{definition}\label{Definition: holomorphic Lipschitz ideal}
    An ideal of holomorphic Lipschitz maps (also called holomorphic Lipschitz ideal) is a subclass $\A^{\HL}$ of $\HL$ such that for any complex Banach spaces $X$ and $Y$, the components
    $$
    \A^{\HL}(B_X,Y) := \HL(B_X,Y) \cap  \A^{\HL}
    $$
    satisfy the following conditions:
\begin{itemize}
  \item[(I1)] $\A^{\HL}(B_X,Y)$ is a linear subspace of $\HL(B_X,Y)$.
  \item[(I2)] Given $h \in \HL(B_X)$ and $y \in Y$, the map $h \cdot y: B_X \to Y$, defined as $(h \cdot y)(x)=h(x)y$ for all $x\in B_X$, is in $\A^{\HL}(B_X,Y)$. 
  \item[(I3)] The ideal property: if $W$ and $Z$ are complex Banach spaces, $f \in \A^{\HL}(B_X,Y)$, $T \in \L(Y,W)$ and $g \in \HL(B_Z,B_X)$, then we have that $T \circ f \circ g \in \A^{\HL}(B_Z,W)$.    
\end{itemize}
If we endow this subclass $\A^{\HL}$ with a function $\left\|\cdot\right\|_{\A^{\HL}}: \A^{\HL} \to \R_0^+$ satisfying the properties:
    \begin{itemize}
        \item[(N1)] $(\A^{\HL}(B_X,Y),\left\|\cdot\right\|_{\A^{\HL}})$ is a ($s$-Banach) $s$-normed space with $s \in (0,1]$ and $L(f) \leq \|f\|_{\A^{\HL}}$ for all $f \in \A^{\HL}(B_X,Y)$,
        \item[(N2)] $\|h \cdot y\|_{\A^{\HL}} = L(h)\|y\|$ for all $h \in \HL(B_X)$ and $y \in Y$,
        \item[(N3)] $\|T \circ f \circ g\|_{\A^{\HL}} \leq \|T\|\|f\|_{\A^{\HL}}L(g)$ if $f \in \A^{\HL}(B_X,Y)$, $T \in \L(Y,W)$ and $g \in \HL(B_Z,B_X)$, 
    \end{itemize}
then we say that $[\A^{\HL},\left\|\cdot\right\|_{\A^{\HL}}]$ is a ($s$-Banach) $s$-normed holomorphic Lipschitz ideal.

If $[\A^{\HL},\left\|\cdot\right\|_{\A^{\HL}}]$ and $[\B^{\HL},\left\|\cdot\right\|_{\B^{\HL}}]$ are $s$-normed holomorphic Lipschitz ideals, we will write 
$$
[\A^{\HL},\left\|\cdot\right\|_{\A^{\HL}}] \leq [\B^{\HL},\left\|\cdot\right\|_{\B^{\HL}}]
$$
whenever $\A^{\HL}\subseteq\B^{\HL}$ and $\|f\|_{\B^{\HL}} \leq \|f\|_{\A^{\HL}}$ for all $f \in \A(B_X,Y)$. 
\end{definition}

Note that this concept of holomorphic Lipschitz ideal is closely related to the notion of hyper-ideal of polynomials since it contains finite-rank holomorphic Lipschitz mappings but this does not happen with polynomial ideals in general.

Influenced by \cite[8.1]{Pie80}, we now introduce a method to construct new $s$-normed holomorphic Lipschitz ideals from given ones.

\begin{definition}
A holomorphic Lipschitz procedure is a correspondence 
$$
[\A^{\HL},\|\cdot\|_{\A^{\HL}}]\mapsto[(\A^{\HL})^{new},\|\cdot\|_{\A^{\HL}}^{new}]
$$
that maps an $s$-normed holomorphic Lipschitz ideal $[\A^{\HL},\|\cdot\|_{\A^{\HL}}]$ to an $s$-normed holomorphic Lipschitz ideal $[(\A^{\HL})^{new},\|\cdot\|_{\A^{\HL}}^{new}]$. If we note
$$
[\A^{\HL},\|\cdot\|_{\A^{\HL}}]^{new}:=[(\A^{\HL})^{new},\|\cdot\|_{\A^{\HL}}^{new}],
$$
a holomorphic Lipschitz procedure is called:
\begin{enumerate}
	\item[(M)] Monotone if $[\A^{\HL},\|\cdot\|_{\A^{\HL}}]\leq [\B^{\HL},\|\cdot\|_{\B^{\HL}}]$ implies $[\A^{\HL},\|\cdot\|_{\A^{\HL}}]^{new}\leq [\B^{\HL},\|\cdot\|_{\B^{\HL}}]^{new}$.
	\item[(I)] Idempotent if $([\A^{\HL},\|\cdot\|_{\A^{\HL}}]^{new})^{new}=[\A^{\HL},\|\cdot\|_{\A^{\HL}}]^{new}$.
  \item[(H)] Hull if it is monotone, idempotent and $[\A^{\HL},\|\cdot\|_{\A^{\HL}}]\leq [\A^{\HL},\|\cdot\|_{\A^{\HL}}]^{new}$.
\end{enumerate}
\end{definition}

\begin{definition}
Let $X$ and $Y$ be complex Banach spaces. For a normed holomorphic Lipschitz ideal $[\A^{\HL},L(\cdot)]$, we set 
$$
(\A^{\HL})^{clos}(B_X,Y)=\left\{f\in\HL(B_X,Y)\;|\; \exists (f_n)_{n=1}^\infty\subseteq\A^{\HL}(B_X,Y) \colon \lim_{n\to\infty} L(f_n-f)=0\right\};
$$
and for an $s$-normed holomorphic Lipschitz ideal $[\A^{\HL},\|\cdot\|_{\A^{\HL}}]$, we define  
$$
(\A^{\HL})^{reg}(B_X,Y)=\left\{f\in\HL(B_X,Y)\colon \kappa_Y\circ f\in\A^{\HL}(B_X,Y^{**})\right\},
$$
where $\kappa_Y: Y \to Y^{**}$ is the canonical isometric linear embedding.
\end{definition}

The following result can be proved easily.

\begin{proposition}
The correspondences 
$$
[\A^{\HL},L(\cdot)]\mapsto[(\A^{\HL})^{clos},L(\cdot)]
$$
and 
$$
[\A^{\HL},\|\cdot\|_{\A^{\HL}}]\mapsto[\A^{\HL},\|\cdot\|_{\A^{\HL}}]^{reg}
$$
are hull holomorphic Lipschitz procedures. The holomorphic Lipschitz spaces $(\A^{\HL})^{clos}(B_X,Y)$ and $(\A^{\HL})^{reg}(B_X,Y)$ are called the closed hull and regular hull of $\A^{\HL}(B_X,Y)$, respectively. $\hfill\qed$
\end{proposition}

The preceding result motivates the introduction of the following concepts.

\begin{definition}
A normed holomorphic Lipschitz ideal $[\A^{\HL},L(\cdot)]$ is called closed if $[\A^{\HL},L(\cdot)]=[(\A^{\HL})^{clos},L(\cdot)]$, that is, every component $\A^{\HL}(B_X,Y)$ is a closed subspace of $\HL(B_X,Y)$ with the topology induced by the Lipschitz norm.

an $s$-normed holomorphic Lipschitz ideal $[\A^{\HL},\|\cdot\|_{\A^{\HL}}]$ is called regular if $[\A^{\HL},\|\cdot\|_{\A^{\HL}}]=[\A^{\HL},\|\cdot\|_{\A^{\HL}}]^{reg}$, that is, if $f\in\HL(B_X,Y)$ and $\kappa_Y \circ f \in \A^{\HL}(B_X,Y^{**})$, then $f\in\A^{\HL}(B_X,Y)$ with $\|f\|_{\A^{\HL}}=\|\kappa_Y \circ f\|_{\A^{\HL}}$.
\end{definition}


\subsection{Composition ideals of holomorphic Lipschitz mappings}\label{Subsection 2}

We now introduce the composition procedure to generate $s$-normed holomorphic Lipschitz ideals.

\begin{definition}\label{Definition: composition procedure}
Let $[\A,\left\|\cdot\right\|_\A]$ be an $s$-normed operator ideal. For any complex Banach spaces $X$ and $Y$, define
$$
\A \circ \HL(B_X,Y) = \left\{f \in \HL(B_X,Y) \;|\; \exists S\in\A(Z,Y),\; \exists h\in\HL(B_X,Z)\colon f=S\circ h\right\},
$$
and 
$$
\|f\|_{\A \circ \HL}=\inf\{\|S\|_\A L(h)\},
$$
where the infimum extends over all factorizations of $f$ as above.
\end{definition}

The belonging of a holomorphic Lipschitz map to $\A \circ \HL$ may be characterized by means of its linearization as follows. This result can be compared to \cite[Theorem 3.2]{Aro10} for bounded holomorphic functions.

\begin{theorem}\label{Theorem: linearization}
Let $[\A,\left\|\cdot\right\|_\A]$ be an $s$-normed operator ideal and $f\in\HL(B_X,Y)$. The following statements are equivalent:
\begin{enumerate}
	\item $f$ belongs to $\A\circ\HL(B_X,Y)$.
	\item $T_f$ belongs to $\A(\G(B_X),Y)$.
\end{enumerate}
In such a case, $\|f\|_{\A \circ \HL} = \|T_f\|_\A$. 
\end{theorem}

\begin{proof}
$(i)\Rightarrow(ii)$: Assume that $f\in\A\circ\HL(B_X,Y)$. Then $f=S\circ h$, where $Z$ is a complex Banach space, $S\in\A(Z,Y)$ and $h\in\HL(B_X,Z)$. By \cite[Proposition 2.3 (c)]{Aro24}, there exist $T_f\in\L(\G(B_X),Y)$ and $T_h\in\L(\G(B_X),Z)$ with $\left\|T_f\right\|=L(f)$ and $\left\|T_h\right\|=L(h)$ such that $f=T_f\circ\delta_X$ and $h=T_h\circ\delta_X$. It follows that $T_f\circ\delta_X=S\circ T_h\circ\delta_X$, and since $\lin(\delta_X(B_X))$ is a dense linear subspace of $\G(B_X)$ by \cite[Proposition 2.3 (b)]{Aro24}, we deduce that $T_f=S\circ T_h$ and $\left\|T_f\right\|_\A\leq\left\|S\right\|_\A L(h)$. By the ideal property of $\A$, we conclude that $T_f\in\A(\G(B_X),Y)$, and passing to the infimum over all the factorizations of $f$ as above, we obtain that $\left\|T_f\right\|_\A\leq\|f\|_{\A \circ\HL}$. 

$(ii)\Rightarrow(i)$: Assume that $T_f\in\A(\G(B_X),Y)$. Notice that $f = T_f \circ \delta_X \in \A \circ \HL(B_X,Y)$ since $\delta_X \in \HL(B_X,\G(B_X))$ with $L(\delta_X) = 1$ by \cite[Proposition 2.3 (a)]{Aro24}. Moreover,
$$
\|f\|_{\A \circ \HL} = \|T_f \circ \delta_X\|_{\A \circ \HL} \leq \|T_f\|_\A L(\delta_X) = \|T_f\|_\A.
$$
\end{proof}

Some properties of an operator ideal $\A$ can be inherited by $\A\circ \HL$ as we see below.

\begin{corollary}\label{Theorem: composition ideal}
If $[\A,\left\|\cdot\right\|_\A]$ is an $s$-normed operator ideal, then $[\A \circ \HL,\left\|\cdot\right\|_{\A \circ \HL}]$ is an $s$-normed holomorphic Lipschitz ideal called the holomorphic Lipschitz composition ideal of $\A$. Moreover, $[\A \circ \HL,\left\|\cdot\right\|_{\A \circ \HL}]$ is closed (resp., $s$-Banach, regular) whenever $[\A,\left\|\cdot\right\|]$ (resp., $[\A,\left\|\cdot\right\|_\A]$) is so. 
\end{corollary}

\begin{proof}
We first prove that $\A\circ \HL(B_X,Y)$ is a linear subspace of $\HL(B_X,Y)$. Let $\lambda\in\mathbb{C}$ and $f,g\in\A\circ\HL(B_X,Y)$. Then $T_f,T_g\in\A(\G(B_X),Y)$ by Theorem \ref{Theorem: linearization}. Hence $T_{\lambda f}=\lambda T_f$ and $T_{f+g}=T_f+T_g$ belongs to $\A(\G(B_X),Y)$ by \cite[Proposition 2.3 (c)]{Aro24} and the ideal property of $\A$. This implies that $\lambda f,f+g\in\A\circ\HL(B_X,Y)$ again by Theorem \ref{Theorem: linearization}, as required. 

We now prove the conditions required in Definition \ref{Definition: holomorphic Lipschitz ideal}. 

(I1-N1): Using Theorem \ref{Theorem: linearization} and \cite[Proposition 2.3 (c)]{Aro24}, we obtain that $\left\|\cdot\right\|_{\A \circ \HL}$ is a $s$-norm on $\A\circ \HL(B_X,Y)$ since 
\[\begin{split}
\left\|\lambda f\right\|_{\A\circ\HL}&=\left\|T_{\lambda f}\right\|_{\A}=\left\|\lambda T_f\right\|_{\A}=\left|\lambda\right|\left\|T_f\right\|_{\A}=\left|\lambda\right|\left\|f\right\|_{\A\circ\HL},\\
\left\|f+g\right\|^s_{\A\circ\HL}&=\left\|T_{f+g}\right\|^s_\A=\left\|T_f+T_g\right\|^s_{\A}\leq \left\|T_f\right\|^s_{\A}+\left\|T_g\right\|^s_{\A}=\left\|f\right\|^s_{\A \circ \HL}+\left\|g\right\|^s_{\A \circ \HL},\\
\left\|f\right\|_{\A \circ \HL}&=0\Rightarrow \left\|T_f\right\|_\A=0\Rightarrow T_f=0\Rightarrow f=T_f\circ\delta_X=0.
\end{split}\]
Using \cite[Proposition 2.3 (c)]{Aro24}, the fact that $[\A,\left\|\cdot\right\|_\A]$ is an $s$-normed operator ideal and Theorem \ref{Theorem: linearization}, we deduce that $L(f)=\left\|T_f\right\|\leq\left\|T_f\right\|_\A=\left\|f\right\|_{\A \circ \HL}$ for all $f\in\A\circ\HL$. 

(I2-N2): Let $h\in\HL(B_X)$ and $y\in Y$. It is easy to prove that $h\cdot y\in\HL(B_X,Y)$ with $L(h\cdot y)=L(h)\left\|y\right\|$. Note that $T_{h\cdot y}=T_h\cdot y$ since $T_h\cdot y\in\L(\G(B_X),Y)$ and
$$
(h\cdot y)(x)=h(x)y=T_h(\delta_X(x))y=(T_h\cdot y)(\delta_X(x))=((T_h\cdot y)\circ\delta_X)(x)
$$
for all $x\in B_X$. By the ideal property of $[\A,\|\cdot\|_\A]$, we have that $T_{h\cdot y}\in\A(\G(B_X),Y)$ with $\left\|T_{h\cdot y}\right\|_\A=\left\|T_h\right\|\left\|y\right\|=L(h)\left\|y\right\|$. Hence $h\cdot y\in\A\circ\HL(B_X,Y)$ with $\left\|h\cdot y\right\|_{\A\circ\HL}=L(h)\left\|y\right\|$ by Theorem \ref{Theorem: linearization}.

(I3-N3): Let $W,Z$ be complex Banach spaces, $f\in\A\circ\HL(B_X,Y)$, $T\in\L(Y,W)$ and $g\in\HL(B_Z,B_X)$. In light of \cite[p. 3033]{Aro24}, we can take $\hat{g}:=T_{\delta_Y\circ g}\in\L(\G(B_Z),\G(B_X))$ such that $\hat{g}\circ\delta_Z=\delta_X\circ g$ with $\|\hat{g}\|=L(g)$. Note that
$$
T\circ f\circ g=T\circ(T_f\circ\delta_X)\circ g=(T\circ T_f\circ\hat{g})\circ\delta_Z,
$$
where $T\circ T_f\circ\hat{g}\in\A(\G(B_Z),W)$ by the ideal property of $\A$ and $\delta_Z\in\HL(B_Z,\G(B_Z))$. Hence, $T\circ f\circ g\in\A\circ\HL(B_Z,W)$ and 
$$
\|T \circ f \circ g\|_{\A \circ \HL}\leq \|T \circ T_f \circ \hat{g}\|_\A L(\delta_Z)\leq \|T\|\|T_f\|_\A\|\hat{g}\|=\|T\|\|f\|_{\A\circ \HL}L(g).
$$

Hence, $[\A\circ \HL,\left\|\cdot\right\|_{\A \circ \HL}]$ is an $s$-normed holomorphic Lipschitz ideal. To see that the $s$-norm $\left\|\cdot\right\|_{\A \circ \HL}$ is complete whenever $\left\|\cdot\right\|_\A$ is so, it suffices to note that $f\mapsto T_f$ is an isometric isomorphism from $(\A \circ \HL(B_X,Y),\left\|\cdot\right\|_{\A \circ \HL})$ onto $(\A(\G(B_X),Y),\left\|\cdot\right\|_\A)$ by \cite[Proposition 2.3 (c)]{Aro24} and Theorem \ref{Theorem: linearization}.

Assume that $[\A,\left\|\cdot\right\|]$ is closed and let us show that so is $[\A \circ \HL,\left\|\cdot\right\|_{\A \circ \HL}]$. Let $f \in \HL(B_X,Y)$ and let $(f_n)$ be a sequence in $\A \circ \HL(B_X,Y)$ satisfying $\lim_{n\to\infty}L(f_n-f)=0$. By Theorem \ref{Theorem: linearization} and \cite[Proposition 2.3 (c)]{Aro24}, we have that $T_{f_n}\in\A(\G(B_X),Y)$ and $\|T_{f_n}-T_f\|=\|T_{f_n-f}\|=L(f_n-f)\to 0$ as $n\to\infty$. Hence $T_f \in \A(\G(B_X),Y)$ and reapplying Theorem \ref{Theorem: linearization} we conclude that $f\in\A\circ\HL(B_X,Y)$.

Finally, to prove the regularity, let us suppose that $[\A,\left\|\cdot\right\|_\A]$ is regular. Let $f \in \HL(B_X,Y)$ and assume that $\kappa_Y \circ f \in \A \circ \HL(B_X,Y^{**})$. Clearly, $T_{\kappa_Y \circ f}=\kappa_Y\circ T_f$. Since $T_{\kappa_Y \circ f}\in\A(\G(B_X),Y^{**})$ by Theorem \ref{Theorem: linearization} and $[\A,\left\|\cdot\right\|_\A]$ is regular, it follows that $T_f\in\A(\G(B_X),Y)$ with $\|T_f\|_\A = \|\kappa_Y \circ T_f\|_\A$. Hence, $f \in \A \circ \HL(B_X,Y)$ and
$$
\|f\|_{\A \circ \HL} = \|T_f\|_\A = \|\kappa_Y \circ T_f\|_\A = \|T_{\kappa_Y \circ f}\|_\A = \|\kappa_Y \circ f\|_{\A \circ \HL}.
$$
\end{proof}


\subsection{Dual ideal of holomorphic Lipschitz mappings}\label{Subsection 3}

Let $[\A,\left\|\cdot\right\|_\A]$ be an $s$-normed operator ideal. Following \cite[4.2]{Pie80}, $[\A^{\mathrm{dual}},\|\cdot\|_{\A^{\mathrm{dual}}}]$ determines an $s$-normed operator ideal, -- called the dual ideal of $\A$ -- , whose components are the spaces 
$$
\A^{\mathrm{dual}}(X,Y) = \{T \in \L(X,Y): T^* \in \A(Y^*,X^*)\}
$$
for any Banach spaces $X$ and $Y$, equipped with the $s$-norm given by
$$
\|T\|_{\A^{\mathrm{dual}}} = \|T^*\|_\A \qquad (T \in \A^{\mathrm{dual}}(X,Y)).
$$

The introduction of a version of this concept for ideals of holomorphic Lipschitz maps requires an equivalent of the notion of adjoint operator in this setting, as follows.

\begin{definition}
Given $f \in \HL(B_X,Y)$, the holomorphic Lipschitz transpose of $f$ is the map $f^t\colon Y^*\to\HL(B_X)$ given by $f^t(y^*)=y^*\circ f$ for all $y^*\in Y^*$.
\end{definition}

To show that $f^t$ satisfies the desired properties, let us recall that $\HL(B_X)$ is isometrically isomorphic to $\G(B_X)^*$ throughout the map $\Lambda_X:\HL(B_X) \to \G(B_X)^*$ defined as
$$
\Lambda_X(g)=T_g\qquad (g\in\HL(B_X)).
$$

\begin{proposition}\label{Proposition: transpose properties}
Let $f \in \HL(B_X,Y)$. Then $f^t \in \L(Y^*,\HL(B_X))$ with $\|f^t\|=L(f)$ and $f^t = \Lambda_X^{-1}\circ (T_f)^*$.
\end{proposition}

\begin{proof}
Given $y^* \in Y^*$, note that $(y^* \circ f)(0)=0$ and $y^*\circ f\in\H^\infty(B_X,Y)$ with $(y^*\circ f)'=y^*\circ f'$ and $\left\|y^*\circ f\right\|_\infty\leq \left\|y^*\right\|\left\|f\right\|_\infty$. Moreover, for all $x,y\in B_X$ with $x\neq y$, we have 
$$
\frac{|(y^*\circ f)(x)-(y^*\circ f)(y)|}{\|x-y\|}
\leq \|y^*\|\frac{\|f(x)-f(y)\|}{\|x-y\|}\leq \|y^*\|L(f),
$$
and thus $y^* \circ f\in \HL(B_X,Y)$ with $L(y^* \circ f)\leq \left\|y^*\right\|L(f)$. It follows that $f^t\in\L(Y^*,\HL(B_X))$ with $\|f^t\|\leq L(f)$. On the other hand, let $\varepsilon \in (0,L(f))$ and consider $x,y \in B_X$ satisfying $\|f(x)-f(y)\|/\|x-y\| > L(f)-\varepsilon$. By the Hahn--Banach Theorem, there is $z^* \in Y^*$ with $\|z^*\|=1$ such that $|z^*(f(x)-f(y))|=\|f(x)-f(y)\|$. Thus,
\[\begin{split}
\|f^t\| &\geq \sup_{y^* \neq 0} \frac{L(f^t(y^*))}{\|y^*\|} \geq \frac{L(z^* \circ f)}{\|z^*\|}\geq \frac{|z^*(f(x))-z^*(f(y))|}{\|x-y\|}\\
&= \frac{\|f(x)-f(y)\|}{\|x-y\|} > L(f)-\varepsilon.
\end{split}\]
Just doing $\varepsilon\to 0$ we conclude that $\|f^t\| = L(f)$. Furthermore, 
\[\begin{split}
\langle(\Lambda_X \circ f^t)(y^*),\delta_X(x)\rangle 
&=\langle \Lambda_X(y^* \circ f),\delta_X(x)\rangle = \langle T_{y^* \circ f},\delta_X(x)\rangle\\
&=\langle y^* \circ T_f,\delta_X(x)\rangle = \langle (T_f)^*(y^*),\delta_X(x)\rangle
\end{split}\]
for all $y^* \in Y^*$ and $x \in B_X$. Since $\G(B_X) = \overline{\lin}(\delta_X(B_X))$, we conclude that $f^t = \Lambda_X^{-1}\circ (T_f)^*$.
\end{proof}

$\HL(B_X,Y)$ can be identified with the subspace of $\L(Y^*,\HL(B_X))$ consisting of all weak*-to-weak* continuous linear operators from $Y^*$ into $\HL(B_X))$.

\begin{corollary}\label{teo-4-1}
The correspondence $f\mapsto f^t$ is an isometric isomorphism from $\HL(B_X,Y)$ onto $\L((Y^*,w^*);(\HL(B_X),w^*))$.
\end{corollary}

\begin{proof}
Let $f\in\HL(B_X,Y)$. Hence $f^t=\Lambda_X^{-1}\circ (T_f)^*\in\L((Y^*,w^*);(\HL(B_X)),w^*)$ by Proposition \ref{Proposition: transpose properties} and \cite[Theorem 3.1.11]{Meg-98}. We have $||f^t||=L(f)$, and in order to prove that the map in the statement is onto, let $T\in\L((Y^*,w^*);(\HL(B_X)),w^*)$. Then $\Lambda_X\circ T\in\L((Y^*,w^*);(\G(B_X)^*,w^*))$ and therefore $S^*=\Lambda_X\circ T$ for some $S\in\L(\G(B_X),Y)$. By \cite[Proposition 2.3 (c)]{Aro24}, there exists $f\in\HL(B_X,Y)$ such that $T_f=S$ and we conclude that $T=\Lambda_X^{-1}\circ(T_f)^*=f^t$. 
\end{proof}

We now introduce a notion of holomorphic Lipschitz dual ideal of an operator ideal as desired.

\begin{definition}
Let $[\A,\left\|\cdot\right\|_\A]$ be an $s$-normed operator ideal. Given complex Banach spaces $X$ and $Y$, define 
$$
\A^{\HL\text{-}\mathrm{dual}}(B_X,Y) = \{f \in \HL(B_X,Y)\colon f^t \in \A(Y^*,\HL(B_X))\},
$$
and 
$$
\|f\|_{\A^{\HL\text{-}\mathrm{dual}}} = \|f^t\|_\A\qquad (f\in \A^{\HL\text{-}\mathrm{dual}}(B_X,Y)).
$$
\end{definition}

The pair $[\A^{\HL\text{-}\mathrm{dual}},\left\|\cdot\right\|_{\A^{\HL\text{-}\mathrm{dual}}}]$ is an $s$-normed operator ideal of composition type called the holomorphic Lipschitz dual ideal of $\A$. To be more precise, we have the following.

\begin{theorem}\label{Theorem: dual ideal}
Let $[\A,\left\|\cdot\right\|_\A]$ be an $s$-normed operator ideal. Then 
$$
[\A^{\HL\text{-}\mathrm{dual}},\left\|\cdot\right\|_{\A^{\HL\text{-}\mathrm{dual}}}] = [\A^{\mathrm{dual}}\circ \HL,\left\|\cdot\right\|_{\A^{\mathrm{dual}}\circ \HL}].
$$
\end{theorem}

\begin{proof}
If $f\in\A^{\HL\text{-}\mathrm{dual}}(B_X,Y)$, then $f^t\in\A(Y^*,\HL(B_X))$. By \cite[Proposition 2.3 (c)]{Aro24}, there exists $T_f\in\L(\G(B_X),Y)$ such that $f=T_f \circ \delta_X$. Since $(T_f)^*=\Lambda_X\circ f^t$ by Proposition \ref{Proposition: transpose properties}, the ideal property of $\A$ allows us to ensure that $(T_f)^*\in\A(Y^*,\G(B_X)^*)$ with $\|(T_f)^*\|_\A \leq\|\Lambda_X\|\|f^t\|_\A = \|f^t\|_\A$. Hence, $T_f\in\A^{\mathrm{dual}}(\G(B_X),Y)$ with $\|T_f\|_{\A^{\mathrm{dual}}} = \|(T_f)^*\|_\A $. We conclude that $f\in\A^{\mathrm{dual}}\circ\HL(B_X,Y)$ with $\|f\|_{\A^{\mathrm{dual}}\circ \HL} = \|T_f\|_{\A^{\mathrm{dual}}}$ by Theorem \ref{Theorem: linearization A-compact} and thus $\|f\|_{\A^{\mathrm{dual}}\circ\HL}\leq\|f\|_{\A^{\HL\text{-}\mathrm{dual}}}$. 

Conversely, if $f\in\A^{\mathrm{dual}}\circ\HL(B_X,Y)$, then $f=T\circ h$, where $T \in \A^{\mathrm{dual}}(Z,Y)$ and $h \in \HL(B_X,Z)$ for some complex Banach space $Z$. Notice that $f^t = h^t \circ T^*$ since 
\[\begin{split}
f^t(y^*) &= (T \circ h)^t(y^*) = y^* \circ (T \circ h) = (y^* \circ T) \circ h\\
&=T^*(y^*)\circ h = h^t(T^*(y^*)) = (h^t \circ T^*)(y^*)
\end{split}\]
for all $y^*\in Y^*$. Hence $f^t\in\A(Y^*,\HL(B_X))$ by the ideal property of $\A$ since $T^*\in\A(Y^*,Z^*)$. Therefore $f\in\A^{\HL\text{-}\mathrm{dual}}(B_X,Y)$, and from the inequality
$$
\|f\|_{\A^{\HL\text{-}\mathrm{dual}}} = \|f^t\|_\A = \|h^t \circ T^*\|_\A \leq \|h^t\|\|T^*\|_\A = L(h)\|T\|_{\A^{\mathrm{dual}}},
$$
we infer that $\|f\|_{\A^{\HL\text{-}\mathrm{dual}}} \leq \|f\|_{\A^{\mathrm{dual}}\circ\HL}$ by taking the infimum over all such representations of $f$.
\end{proof}


\section{$\A$-Compact holomorphic Lipschitz mappings}\label{Section 2}

Following the comments in \cite[p. 1094]{LasTur18}, the original definitions of $\A$-compactness and its measure $m_\A$ were given in \cite{CarSte84,LasTur13} for normed operator ideals, but they can be easily extended for $s$-normed operator ideals with $s\in (0,1]$ and their properties also hold with slight modifications.

Given an $s$-normed operator ideal $\A$, our aim in this section is to introduce a variant for holomorphic Lipschitz maps of the concept of $\A$-compact linear operators between Banach spaces \cite{CarSte84,LasTur13}. Towards this end, we denote the Lipschitz image of a map $f\colon B_X\to Y$ by 
$$
\Im_{L}(f):=\left\{\frac{f(x)-f(y)}{\|x-y\|}\colon x,y \in B_X,\, x\neq y\right\}\subseteq Y.
$$ 
Notice that $f$ is Lipschitz if and only if $\Im_{L}(f)$ is a bounded subset of $Y$. This motivates the following.

\begin{definition}\label{Definition: A-compact holomorphic Lipschitz}
Let $[\A,\left\|\cdot\right\|_\A]$ be an $s$-normed operator ideal. Given complex Banach spaces $X$ and $Y$, a map $f\in\HL(B_X,Y)$ is said to be $\A$-compact holomorphic Lipschitz if $\Im_L(f)$ is a relatively $\A$-compact subset of $Y$. If we denote by $\HL^{\K_\A}(B_X,Y)$ the set of all $\A$-compact holomorphic Lipschitz maps from $B_X$ into $Y$, we define
$$
\|f\|_{\HL^{\K_\A}} = m_\A(\Im_{L}(f)) \qquad (f \in \HL^{\K_\A}(B_X,Y)).
$$
\end{definition}

The following characterizations of $\A$-compact holomorphic Lipschitz maps are immediate in light of \cite[Lemma 1.1 and Theorem 1.1]{CarSte84}. 

Let us recall (see \cite[Definition 1.2 and Lemma 1.2]{CarSte84}) that a sequence $(x_n)_{n=1}^\infty$ in $X$ is called $\A$-null if there is a Banach space $Y$, an operator $T \in \A(Y,X)$ and a null sequence $(y_n)_{n=1}^\infty$ in $Y$ satisfying that $x_n = T(y_n)$ for all $n \in \mathbb{N}$.

\begin{proposition}\label{Theorem: characterizations A-compact}
Let $[\A,\left\|\cdot\right\|_\A]$ be an $s$-normed operator ideal. The following are equivalent:
\begin{itemize}
\item[(i)] $f \in \HL^{\K_\A}(B_X,Y)$.
\item[(ii)] There are a complex Banach space $Z$ and an operator $T \in \A(Z,Y)$ so that for every $\varepsilon > 0$, there exists $(y^{\varepsilon }_j)_{j=1}^{k_\varepsilon}$ in $Y$ such that $\Im_{L}(f) \subseteq \bigcup_{j=1}^{k_\varepsilon} \{y^{\varepsilon}_j + \varepsilon T(\overline{B}_Z)\}$.
\item[(iii)] There is an $\A$-null sequence $(y_n)_{n=1}^\infty$ in $Y$ so that $\Im_{L}(f)\subseteq\co((y_n)_{n=1}^\infty)$.
\end{itemize}
$\hfill \square$
\end{proposition}

We also omit the proof of the following since it is a direct application of \cite[Theorem 1.2, Proposition 1.8 and Corollary 1.9]{LasTur13}. 

\begin{proposition}\label{Theorem: characterizations A-compact-2}
Let $[\A,\left\|\cdot\right\|_\A]$ be an $s$-normed operator ideal. The following are equivalent:
\begin{itemize}
\item[(i)] $f \in \HL^{\K_\A}(B_X,Y)$.
\item[(ii)] There are a complex Banach space $Z$, two operators $T \in \A(Z,Y)$ and $R \in \overline{\mathcal{F}}(\ell_1,Z)$ and a relatively compact set $K\subseteq\overline{B}_{\ell_1}$ so that $\Im_{L}(f) \subseteq (T \circ R)(K)$.  
\item[(iii)] There are an operator $T \in \A(\ell_1,Y)$ and a relatively compact set $K\subseteq\overline{B}_{\ell_1}$ so that $\Im_{L}(f) \subseteq T(K)$. 
\item[(iv)] $f \in \HL^{\K_{\A \circ \overline{\mathcal{F}}}}(B_X,Y)$.
\end{itemize}
In such a case, we have 
\[\begin{split}	
\|f\|_{\HL^{\K_\A}} &= \inf\{\|T\|_\A\left\|R\right\|\colon T \in \A(Z,Y),\, R \in \overline{\mathcal{F}}(\ell_1,Z),\, K \subseteq \overline{B}_{\ell_1},\, \Im_{L}(f) \subseteq (T\circ R)(K)\}\\
&=\inf\{\|T\|_\A\colon T \in \A(\ell_1,Y),\, K \subseteq \overline{B}_{\ell_1},\, \Im_{L}(f) \subseteq T(K)\}\\
&=m_{\A \circ \overline{\mathcal{F}}}(\Im_{L}(f)).
\end{split}\]
$\hfill \square$
\end{proposition}

Another useful characterization of $\A$-compact holomorphic Lipschitz mappings can be established in terms of their linearizations. Compare it to its polynomial version \cite[Proposition 1]{Tur16}.

\begin{theorem}\label{Theorem: linearization A-compact}
Let $[\A,\left\|\cdot\right\|_\A]$ be an $s$-normed operator ideal and $f\in\HL(B_X,Y)$. The following are equivalent:
\begin{enumerate}
	\item $f$ belongs to $\HL^{\K_\A}(B_X,Y)$.
	\item $T_f$ belongs to $\K_\A(\G(B_X),Y)$.
\end{enumerate}
In such a case, $\|f\|_{\HL^{\K_\A}} = \|T_f\|_{\K_\A}$. As a consequence, $f \mapsto T_f$ is an isometric isomorphism from $(\HL^{\K_\A}(B_X,Y),\|\cdot\|_{\HL^{\K_\A}})$ onto $(\K_\A(\G(B_X),Y),\left\|\cdot\right\|_{\K_\A})$.
\end{theorem}

\begin{proof} 
Note first that an application of \cite[Propositions 2.3 and 2.6]{Aro24} yields the relations:
\[\begin{split}
\Im_{L}(f)=T_f(\Im_{L}(\delta_X)) &\subseteq T_f(\overline{\abco}(\Im_{L}(\delta_X)))=T_f(\overline{B}_{\G(B_X)})\\
                                  &\subseteq \overline{\abco}(T_f(\Im_{L}(\delta_X))) = \overline{\abco}(\Im_{L}(f)).
\end{split}\]

$(i)\Rightarrow(ii)$: If $f\in\HL^{\K_\A}(B_X,Y)$, then $\Im_{L}(f)$ is a relatively $\A$-compact subset of $Y$. A look at \cite[p. 1094]{LasTur18} shows that $\overline{\abco}(\Im_{L}(f))$ is relatively $\A$-compact in $Y$ with $m_\A(\overline{\abco}(\Im_{L}(f)))=m_\A(\Im_{L}(f))$. Hence there exist a complex Banach space $Z$, an operator $T \in \A(Z,X)$ and a compact set $M\subseteq Z$ such that $\overline{\abco}(\Im_{L}(f)) \subseteq T(M)$. By the second inclusion in the above chain of relations, we obtain that $T_f(\overline{B}_{\G(B_X)}) \subseteq T(M)$ and thus $T_f(\overline{B}_{\G(B_X)})$ is relatively $\A$-compact in $Y$ and $m_\A(T_f(\overline{B}_{\G(B_X)})) \leq m_\A(\overline{\abco}(\Im_{L}(f)))$. Therefore $T_f \in \K_\A(\G(B_X),Y)$ and
    $$
    \|T_f\|_{\K_\A} = m_\A(T_f(\overline{B}_{\G(B_X)})) \leq m_\A(\Im_{L}(f)) = \|f\|_{\HL^{\K_\A}}.
    $$

$(ii)\Rightarrow(i)$: Assume that $T_f \in \K_\A(\G(B_X),Y)$, that is, $T_f(\overline{B}_{\G(B_X)})$ is a relatively $\A$-compact subset of $Y$. Now, by the first inclusion in the cited chain, it follows that $\Im_{L}(f)$ is relatively $\A$-compact in $Y$ with $m_\A(\Im_{L}(f)) \leq m_\A(T_f(\overline{B}_{\G(B_X)}))$, that is, $f \in \HL^{\K_{\A}}(B_X,Y)$ and $\|f\|_{\HL^{\K_\A}}\leq \|T_f\|_{\K_\A}$.

Applying what we have just proved and \cite[Proposition 2.3 (c)]{Aro24}, the last assertion of the statement is deduced.
\end{proof}

The combination of Theorems \ref{Theorem: linearization} and \ref{Theorem: linearization A-compact} shows that $[\HL^{\K_\A},\|\cdot\|_{\HL^{\K_\A}}]$ is an $s$-normed holomorphic Lipschitz ideal of composition type which inherits the properties of $[\K_\A,\left\|\cdot\right\|_{\K_\A}]$ by Corollary \ref{Theorem: composition ideal}.

\begin{corollary}\label{Corollary: A-compact composition}
Let $[\A,\left\|\cdot\right\|_\A]$ be an $s$-normed operator ideal. Then 
$$
[\HL^{\K_\A},\left\|\cdot\right\|_{\HL^{\K_\A}}]=[\K_\A\circ\HL,\left\|\cdot\right\|_{\K_\A\circ\HL}].
$$
Moreover, $[\HL^{\K_\A},\left\|\cdot\right\|_{\HL^{\K_\A}}]$ is $s$-Banach (resp., regular) whenever $[\K_\A,\left\|\cdot\right\|_{\K_\A}]$ is so. $\hfill \square$
\end{corollary}

Let $[\A,\left\|\cdot\right\|_\A]$ be an $s$-normed operator ideal. Recall by \cite[4.7 and 8.5]{Pie80} that the surjective hull of $\A$ is defined for any Banach spaces $X$ and $Y$ by
$$
\A^{sur}(X,Y)=\{T\in \L(X,Y)\colon T \circ Q_X \in \A(\ell_1(\overline{B}_X),Y)\},
$$
where $Q_X\colon \ell_1(\overline{B}_X) \to X$ is the canonical metric surjection defined by 
$$
Q_X((\lambda_x)_{x\in \overline{B}_X})=\sum_{x\in\overline{B}_X}\lambda_x x\qquad \left((\lambda_x)_{x\in \overline{B}_X}\in\ell_1(\overline{B}_X)\right).
$$
Moreover, $[\A^{sur},\|\cdot\|_{\A^{sur}}]$ generates an $s$-normed operator ideal if we set
$$
\|T\|_{\A^{sur}}=\|T\circ Q_X\|_\A \qquad (T\in\A^{sur}(X,Y)).
$$

More descriptions of $\A$-compact holomorphic Lipschitz mappings can be given by composition with both $\K_{\A\circ\overline{\mathcal{F}}}$ and $(\A\circ\overline{\mathcal{F}})^{sur}$.

\begin{corollary}
Let $[\A,\left\|\cdot\right\|_\A]$ be a $s$-Banach operator ideal. Then
\[\begin{split}
[\HL^{\K_\A},\left\|\cdot\right\|_{\HL^{\K_\A}}]&=[\K_{\A\circ\overline{\mathcal{F}}}\circ\HL,\left\|\cdot\right\|_{\K_{\A\circ\overline{\mathcal{F}}}\circ\HL}]\\
&=[(\A\circ\overline{\mathcal{F}})^{sur}\circ\HL,\left\|\cdot\right\|_{(\A\circ\overline{\mathcal{F}})^{sur}\circ\HL}].
\end{split}\]
\end{corollary}

\begin{proof}
Let $f\in\HL(B_X,Y)$. Just applying Theorem \ref{Theorem: linearization A-compact}, Theorem \ref{Theorem: linearization} and \cite[Proposition 2.1]{LasTur13}, we have
\[\begin{split}
f \in \HL^{\K_\A}(B_X,Y) &\Leftrightarrow T_f \in \K_\A(\G(B_X),Y)\\
                         &\Leftrightarrow f \in \K_\A \circ \HL(B_X,Y)\\
												 &\Leftrightarrow f \in \K_{\A\circ\overline{\mathcal{F}}}\circ \HL(B_X,Y)\\
                         &\Leftrightarrow f \in (\A \circ \overline{\mathcal{F}})^{sur}\circ \HL(B_X,Y),
\end{split}\]
and
$$
\|f\|_{\HL^{\K_\A}} = \|T_f\|_{\K_\A} = \|f\|_{\K_\A \circ \HL} =\|f\|_{\K_{\A\circ\overline{\mathcal{F}}}\circ \HL}= \|f\|_{(\A \circ \overline{\mathcal{F}})^{sur}\circ \HL}.
$$
\end{proof}

Applying Corollary \ref{Corollary: A-compact composition}, \cite[Corollary 2.4]{LasTur13} and Theorem \ref{Theorem: dual ideal}, we also can describe the holomorphic Lipschitz ideal $\HL^{\K_\A}$ as a dual ideal as follows. A version of this result was established in the polynomial setting in \cite[Proposition 2]{Tur16}. 

\begin{corollary}
Let $[\A,\left\|\cdot\right\|_\A]$ be an $s$-normed operator ideal. Then 
\[\begin{split}
[\HL^{\K_\A},\left\|\cdot\right\|_{\HL^{\K_\A}}]&=[\K_\A\circ\HL,\left\|\cdot\right\|_{\K_\A\circ\HL}]\\
&=[\left(\K_\A^{\mathrm{dual}}\right)^{\mathrm{dual}}\circ\HL,\left\|\cdot\right\|_{\left(\K_\A^{\mathrm{dual}}\right)^{\mathrm{dual}}\circ\HL}]\\
&=[(\K_\A^{\mathrm{dual}})^{\HL\text{-}\mathrm{dual}},\left\|\cdot\right\|_{(\K_\A^{\mathrm{dual}})^{\HL\text{-}\mathrm{dual}}}].
\end{split}\] $\hfill \square$
\end{corollary}

We now provide some characterizations of $\K$-compact holomorphic Lipschitz maps in terms of the continuity and compactness of their transposes. 

From \cite[p. 2455]{LasTur13}, it is well known that the relatively $\K$-compact and the relatively $\overline{\F}$-compact sets coincide with the relatively compact sets and, consequently, $[\K_\A,\|\cdot\|_{\K_\A}]=[\K,\|\cdot\|]$ if $\A=\K,\overline{\F}$. For this reason, the elements of 
$$
[\HL^{\K_\K},\left\|\cdot\right\|_{\HL^{\K_\K}}]:=[\HL^{\K},L(\cdot)]
$$
will be referred to as compact holomorphic Lipschitz maps. The characterizations given below should be compared with \cite[Proposition 2.5]{LasTur13}. Note that the first equivalence in the next result provides a version of Schauder Theorem for holomorphic Lipschitz maps.

\begin{theorem}\label{teo-4-2}
Let $f\in\HL(B_X,Y)$. The following statements are equivalent:
\begin{enumerate}
  \item $f\colon B_X\to Y$ is compact holomorphic Lipschitz.
	\item $f^t\colon Y^*\to\HL(B_X)$ is compact.
	\item $f^t\colon Y^*\to\HL(B_X)$ is bounded-weak*-to-norm continuous. 
	\item $f^t\colon Y^*\to\HL(B_X)$ is compact and bounded-weak*-to-weak continuous. 
	\item $f^t\colon Y^*\to\HL(B_X)$ is compact and weak*-to-weak continuous.
\end{enumerate}
\end{theorem}

\begin{proof}
$(i)\Leftrightarrow (ii)$: Applying Theorem \ref{Theorem: linearization A-compact}, the Schauder Theorem, Proposition \ref{Proposition: transpose properties} and the ideal property of $\K$, one has 
\[\begin{split}
f\in\HL^{\K}(B_X,Y)
&\Leftrightarrow T_f\in\K(\G(B_X),Y)\\
&\Leftrightarrow (T_f)^*\in\K(Y^*,\G(B_X)^*)\\
&\Leftrightarrow f^t=\Lambda_X^{-1}\circ(T_f)^*\in\K(Y^*,\HL(B_X)).
\end{split}\]

$(i)\Leftrightarrow (iii)$: Similarly, 
\[\begin{split}
f\in\HL^{\K}(B_X,Y)
&\Leftrightarrow T_f\in\K(\G(B_X),Y)\\
&\Leftrightarrow (T_f)^*\in\L((Y^*,bw^*);\G(B_X)^*)\\
&\Leftrightarrow f^t=\Lambda_X^{-1}\circ(T_f)^*\in\L((Y^*,bw^*);\HL(B_X)),
\end{split}\]
by Theorem \ref{Theorem: linearization A-compact} and \cite[Theorem 3.4.16]{Meg-98}. 

$(iii)\Leftrightarrow (iv)\Leftrightarrow (v)$: It follows from \cite[Proposition 3.1]{Kim-13}.
\end{proof}

In the line of Corollary \ref{teo-4-1}, the transposition now identifies $\HL^{\K}(B_X,Y)$ with the space of all bounded-weak*-to-norm continuous linear operators from $Y^*$ into $\HL(B_X)$.

\begin{proposition}\label{cor-4-1}
The correspondence $f\mapsto f^t$ is an isometric isomorphism from $\HL^{\K}(B_X,Y)$ onto $\L((Y^*,bw^*);\HL(B_X))$.
\end{proposition}

\begin{proof}
Let $f\in\HL^{\K}(B_X,Y)$. Then $f^t\in\L((Y^*,bw^*);\HL(B_X))$ by Theorem \ref{teo-4-2} and $||f^t||=L(f)$ by Proposition \ref{Proposition: transpose properties}. To prove the surjectivity, take $T\in\L((Y^*,bw^*);\HL(B_X))$. Then $\Lambda_X\circ T\in\L((Y^*,bw^*);\G(B_X)^*)$. Then $\kappa_{\G(B_X)}(\gamma)\circ\Lambda_X\circ T\in\L((Y^*,bw^*);\C)$ for all $\gamma\in\G(B_X)$ and, by \cite[Theorem 2.7.8]{Meg-98}, $\kappa_{\G(B_X)}(\gamma)\circ\Lambda_X\circ T\in\L((Y^*,w^*);\C)$ for all $\gamma\in\G(B_X)$, that is, $\Lambda_X\circ T\in\L((Y^*,w^*);(\G(B_X)^*,w^*))$ by \cite[Corollary 2.4.5]{Meg-98}. Hence $\Lambda_X\circ T=S^*$ for some $S\in\L(\G(B_X),Y)$. 
Note that $S^*\in\L((Y^*,bw^*);\G(B_X)^*)$ and this means that $S\in\K(\G(B_X),Y)$.  
Now, $S=T_f$ for some $f\in\HL^{\K}(B_X,Y)$ by Theorem \ref{Theorem: linearization A-compact}. Finally, we have $T=\Lambda_X^{-1}\circ S^*=\Lambda_X^{-1}\circ (T_f)^*=f^t$.
\end{proof}

For $p\in [1,\infty)$, consider the ideals $\K_p$, $\N_p$ and $\QN_p$ of $p$-compact operators, right $p$-nuclear operators and quasi $p$-nuclear operators between Banach spaces, respectively (see \cite{DelPiSer-10,GalLasTur-12,Rya02,SinKar02} for definitions and properties).

From \cite[Remarks 1.3 and 1.7]{LasTur13}, note that the relatively $\N_p$-compact sets and the relatively $\K_p$-compact sets coincide with the relatively $p$-compact sets, and thus $[\K_\A,\|\cdot\|_{\K_\A}]=[\K_p,\|\cdot\|_{\K_p}]$ if $\A=\N_p,\K_p$. 

The application of Theorem \ref{Theorem: linearization A-compact}, \cite[Corollary 2.7]{GalLasTur-12} and Proposition \ref{Proposition: transpose properties} with the ideal property of $\QN_p$ provides the following characterizations of $\N_p$-compact holomorphic Lipschitz maps. 

\begin{proposition}\label{teo-4-22}
Let $f\in\HL(B_X,Y)$. The following assertions are equivalent:
\begin{enumerate}
  \item $f\in\HL^{\K_p}(B_X,Y)$.
	\item $T_f\in\K_p(\G(B_X),Y)$.
	\item $(T_f)^*\in\QN_p(Y^*,\G(B_X)^*)$.
	\item $f^t\in\QN_p(Y^*,\HL(B_X))$.
\end{enumerate}
In such a case, $\|f\|_{\HL^{\K_p}}=\|T_f\|_{\K_p}=\|(T_f)^*\|_{\QN_p}=\|f^t\|_{\QN_p}$ for all $f\in\HL^{\K_p}(B_X,Y)$. $\hfill \square$
\end{proposition}


\section{$\A$-Bounded holomorphic Lipschitz mappings}\label{Section 3}

Given a Banach operator ideal $[\mathcal{A},\left\|\cdot\right\|_\A]$, the property of $\A$-compactness due to Carl and Stephani \cite{CarSte84} corresponds to a particular case of the property of $\A$-boundedness introduced by Stephani \cite{Ste-80} (see also \cite{GonGut-00}).

For a Banach space $X$, a set $B\subseteq X$ is said to be $\A$-bounded if there exist a Banach space $Z$ and an operator $T\in\A(Z,Y)$ such that $B\subseteq T(\overline{B}_Z)$. We denote by $\mathcal{C}_\A(X)$ the collection of all $\A$-bounded subsets $B$ of $X$. In such a case, a measure of the size of $\A$-boundedness of $B$ can be given by $n_\A(K) = \inf\{\|T\|_\A\}$, where the infimum is taken over all such $Z$ and $T$ as above. 

Furthermore, an operator $T\in\L(X,Y)$ is said to be $\A$-bounded whenever $T(\overline{B}_X)$ is an $\A$-bounded subset of $Y$. The set of all such operators is denoted as $\mathcal{C}_\A(X,Y)$, and $[\mathcal{C}_\A,\|\cdot\|_{\mathcal{C}_\A}]$ is a Banach operator ideal with the norm
$$
\|T\|_{\mathcal{C}_\A} = n_\A(T(\overline{B}_X)) \qquad (T \in \mathcal{C}_\A(X,Y)).
$$

The version of $\A$-bounded linear operators for holomorphic Lipschitz maps reads as follows.

\begin{definition}\label{Definition: A-bounded holomorphic Lipschitz}
Let $[\A,\left\|\cdot\right\|_\A]$ be an $s$-normed operator ideal. Given complex Banach spaces $X$ and $Y$, a map $f\in\HL(B_X,Y)$ is said to be $\A$-bounded holomorphic Lipschitz if $\Im_L(f)$ is an $\A$-bounded subset of $Y$. If $\HL^{\mathcal{C}_\A}(B_X,Y)$ stands for the set of all $\A$-bounded holomorphic Lipschitz maps from $B_X$ into $Y$, we set
$$
\|f\|_{\HL^{\mathcal{C}_\A}} = n_\A(\Im_{L}(f)) \qquad (f \in \HL^{\mathcal{C}_\A}(B_X,Y)).
$$
\end{definition}

A proof similar to that Theorem \ref{Theorem: linearization A-compact} provides us the following result. See \cite{GonGut-00} for a complete list of closed operator ideals which are surjective.

\begin{theorem}\label{Theorem: linearization A-bounded}
Let $[\A,\left\|\cdot\right\|_\A]$ be a closed (in the topology of the operator norm) $s$-normed surjective ideal and $f\in\HL(B_X,Y)$. The following are equivalent:
\begin{enumerate}
	\item $f$ belongs to $\HL^{\mathcal{C}_\A}(B_X,Y)$.
	\item $T_f$ belongs to $\mathcal{C}_\A(\G(B_X),Y)$.
\end{enumerate}
In such a case, $\|f\|_{\HL^{\mathcal{C}_\A}} = \|T_f\|_{\mathcal{C}_\A}$. As a consequence, $f \mapsto T_f$ is an isometric isomorphism from $(\HL^{\mathcal{C}_\A}(B_X,Y),\|\cdot\|_{\HL^{\mathcal{C}_\A}})$ onto $(\mathcal{C}_\A(\G(B_X),Y),\left\|\cdot\right\|_{\mathcal{C}_\A})$.$\hfill \square$
\end{theorem}

\begin{proof} 
From the proof of Theorem \ref{Theorem: linearization A-compact}, one has that  
$$
\Im_{L}(f)\subseteq T_f(\overline{B}_{\G(B_X)})\subseteq\overline{\abco}(\Im_{L}(f)).
$$

$(i)\Rightarrow(ii)$: If $f\in\HL^{\mathcal{C}_\A}(B_X,Y)$, then $\Im_{L}(f)$ is an $\A$-bounded subset of $Y$. By \cite[Proposition 3]{GonGut-00}, $\overline{\abco}(\Im_{L}(f))$ is $\A$-bounded in $Y$ with $n_\A(\overline{\abco}(\Im_{L}(f)))=n_\A(\Im_{L}(f))$. Hence we can take a complex Banach space $Z$ and an operator $T \in \A(Z,X)$ such that $\overline{\abco}(\Im_{L}(f))\subseteq T(B_Z)$. Now, the second inclusion above yields that $T_f(\overline{B}_{\G(B_X)}) \subseteq T(B_Z)$, and thus $T_f(\overline{B}_{\G(B_X)})$ is $\A$-bounded in $Y$ with $n_\A(T_f(\overline{B}_{\G(B_X)})) \leq n_\A(\overline{\abco}(\Im_{L}(f)))$. We conclude that $T_f \in\mathcal{C}_\A(\G(B_X),Y)$ and
$$
\|T_f\|_{\mathcal{C}_\A} = n_\A(T_f(\overline{B}_{\G(B_X)})) \leq n_\A(\Im_{L}(f)) = \|f\|_{\HL^{\mathcal{C}_\A}}.
$$

$(ii)\Rightarrow(i)$: If $T_f \in \mathcal{C}_\A(\G(B_X),Y)$, then $T_f(\overline{B}_{\G(B_X)})$ is an $\A$-bounded subset of $Y$. Now, the first inclusion above gives that $\Im_{L}(f)$ is $\A$-bounded in $Y$ with $n_\A(\Im_{L}(f)) \leq n_\A(T_f(\overline{B}_{\G(B_X)}))$, that is, $f \in \HL^{\mathcal{C}_{\A}}(B_X,Y)$ and $\|f\|_{\HL^{\mathcal{C}_\A}}\leq \|T_f\|_{\mathcal{C}_\A}$.
\end{proof}

Since $\W$-bounded sets coincide with weakly compact sets, the elements of $\HL^{\mathcal{C}_\W}(B_X,Y):=\HL^{\W}(B_X,Y)$ may be called weakly compact holomorphic Lipschitz.  

The following result provides versions of Gantmacher, Gantmacher--Nakamura and Davis--Figiel--Johnson--Pe\l czy\'nski Theorems for weakly compact holomorphic Lipschitz mappings. 

\begin{theorem}\label{teo-4-3}
Let $f\in\HL(B_X,Y)$. The following are equivalent:
\begin{enumerate}
  \item $f\colon B_X\to Y$ is weakly compact holomorphic Lipschitz.
	\item $T_f\colon\G(B_X)\to Y$ is weakly compact.
	\item $f^t\colon Y^*\to\HL(B_X)$ is weakly compact.
	\item $f^t\colon Y^*\to\HL(B_X)$ is weak*-to-weak continuous.
	\item There exist a reflexive Banach space $Z$, an operator $T\in\L(Z,Y)$ and a mapping $g\in\HL(B_X,Z)$ such that $f=T\circ g$.
\end{enumerate}
\end{theorem}

\begin{proof}
$(i)\Leftrightarrow (ii)$ follows from Theorem \ref{Theorem: linearization A-bounded}; $(ii)\Leftrightarrow (iii)$ by Gantmacher Theorem, Proposition \ref{Proposition: transpose properties} and the ideal property of $W$; and $(iii)\Leftrightarrow (iv)$ by Gantmacher--Nakamura Theorem. 

If (ii) holds, then the Davis--Figiel--Johnson--Pe\l czy\'nski Theorem provides a reflexive Banach space $Z$ and operators $R\in\L(Z,Y)$ and $S\in\L(\G(B_X),Z)$ such that $T_f=R\circ S$. By \cite[Proposition 2.3 (c)]{Aro24}, we can find a map $g\in\HL(B_X,Z)$ such that $S=T_g$. Hence $T_f=R\circ T_g=T_{R\circ g}$, this implies that $f=R\circ g$, and this proves (v). Finally, (v) implies (i) because $\Im_L(f)=R(\Im_L(g))$, where $R$ is weak-to-weak continuous by \cite[Theorem 2.5.11]{Meg-98} and $\Im_L(g)$ is relatively weakly compact in $Y$ since it is a bounded subset of the reflexive Banach space $Y$. 
\end{proof}

Motivated by the closed surjective ideals of Rosenthal, Banach--Saks or Asplund operators (see \cite{GonGut-00} and the references therein for the definitions and main properties), we introduce the associate holomorphic Lipschitz maps. 

\begin{definition}
Given complex Banach spaces $X$ and $Y$, a map $f\in\HL(B_X,Y)$ is said to be Rosenthal (resp., Banach--Saks, Asplund) if $\Im_L(f)$ is a Rosenthal (resp., Banach--Saks, Asplund) subset of $Y$.  
\end{definition}

With a proof similar to that of Theorem \ref{teo-4-3}, we may obtain the following.

\begin{theorem}\label{teo-4-3-1}
Let $f\in\HL(B_X,Y)$. The following are equivalent:
\begin{enumerate}
  \item $f\colon B_X\to Y$ is a Rosenthal (resp., Banach--Saks, Asplund) holomorphic Lipschitz map.
	\item $T_f\colon\G(B_X)\to Y$ is a Rosenthal (resp., Banach--Saks, Asplund) operator.
	\item There exist a Banach space $Z$ which contains a copy of $\ell_1$ (resp., has the Banach--Saks property, is Asplund), an operator $T\in\L(Z,Y)$ and a mapping $g\in\HL(B_X,Z)$ such that $f=T\circ g$.
\end{enumerate} $\hfill \square$
\end{theorem}

We next identify $\HL^\W(B_X,Y)$ with the space of all weak*-to-weak continuous linear operators from $Y^*$ into $\HL(B_X)$.

\begin{proposition}\label{cor-4-1b} 
The correspondence $f\mapsto f^t$ is an isometric isomorphism from $\HL^\W(B_X,Y)$ onto $\L((Y^*,w^*);(\HL(B_X),w))$.
\end{proposition}

\begin{proof}
We only need to show that such map is surjective. Let $T\in\L((Y^*,w^*);(\HL(B_X),w))$. Then $\Lambda_X\circ T\in\L((Y^*,w^*);(\G(B_X)^*,w))$ 
and this last set is contained in $\L((Y^*,w^*);(\G(B_X)^*,w^*))$. It follows that $\Lambda_X\circ T=S^*$ for some $S\in\L(\G(B_X),Y)$. 
Hence $S^*\in\L((Y^*,w^*);(\G(B_X)^*,w))$ and, by the Gantmacher--Nakamura Theorem, $S\in\W(\G(B_X),Y)$. Now, $S=T_f$ for some $f\in\HL^{\W}(B_X,Y)$ by Theorem \ref{teo-4-3}. Finally, $T=\Lambda_X^{-1}\circ S^*=\Lambda_X^{-1}\circ (T_f)^*=f^t$, as desired.
\end{proof}

In some cases we also can identify the holomorphic little Lipschitz space $\Hl(B_X,Y)$ with the space of bounded-weak*-to-norm continuous linear operators from $Y^*$ to $\Hl(B_X)$. In its proof we will apply the following criterion for compactness in $\Hl(B_X)$.

\begin{lemma}\label{compact-little}
Let $X$ be a finite dimensional complex Banach space. A set $K\subseteq\Hl(B_X)$ is compact if and only if it is closed, bounded, and satisfies
$$
\lim_{\|x-y\|\to 0}\sup_{f\in K}\frac{|f(x)-f(y)|}{\|x-y\|}=0.
$$
\end{lemma}

\begin{proof}
Suppose that $K\subseteq\Hl(B_X)$ is compact. Clearly, $K$ is closed and bounded. Let $\varepsilon>0$ and we can assume that $K\subseteq \cup_{k=1}^n B(f_k,\varepsilon/2)$ with $f_1,\ldots,f_n\in K$ for some $n\in\mathbb{N}$. There exists $\delta>0$ such that 
$$
\frac{|f_k(x)-f_k(y)|}{\|x-y\|}<\frac{\varepsilon}{2}
$$
whenever $x,y\in B_X$ with $0<\|x-y\|<\delta$ and $k\in\{1,\ldots,n\}$. If $f\in K$, we have that $L(f-f_k)<\varepsilon/2$ for some $k\in\{1,\ldots,n\}$ and thus
$$
\frac{|f(x)-f(y)|}{\|x-y\|}\leq L(f-f_k)+\frac{|f_k(x)-f_k(y)|}{\|x-y\|}<\varepsilon
$$
if $x,y\in B_X$ with $0<\|x-y\|<\delta$, and this establishes that the limit in the statement is 0.

Conversely, assume that $K$ is a closed bounded subset of $\Hl(B_X)$ such that 
$$
\lim_{\|x-y\|\to 0}\sup_{f\in K}\frac{|f(x)-f(y)|}{\|x-y\|}=0.
$$
Let $(f_n)$ be a sequence in $K$. By Montel's Theorem and \cite[Lemma 2.5]{JimVil-13} there exists a subsequence $(f_{n_k})$ which converges pointwise on $B_X$ (in fact, uniformly on compact subsets of $B_X$) to some function $f\in\HL(B_X)$. Let $\varepsilon>0$. There is a $\delta>0$ such that for all $k\in\mathbb{N}$,
$$
\frac{|f_{n_k}(x)-f_{n_k}(y)|}{\|x-y\|}<\frac{\varepsilon}{2}
$$
if $x,y\in B_X$ with $0<\|x-y\|<\delta$, taking limits with $k\to\infty$ yields  
$$
\frac{|f(x)-f(y)|}{\|x-y\|}\leq\frac{\varepsilon}{2}
$$
if $x,y\in B_X$ with $0<\|x-y\|<\delta$, and thus 
$$
\frac{|(f_{n_k}-f)(x)-(f_{n_k}-f)(y)|}{\|x-y\|}<\varepsilon
$$
if $x,y\in B_X$ with $0<\|x-y\|<\delta$. On the other hand, since the set 
$$
K_\delta=\{(x,y)\in B_X\times B_X\colon \|x-y\|\geq\delta\}
$$
is compact in $X^2$ with the product topology, there exists $k_0\in\mathbb{N}$ such that if $k\geq k_0$, then 
$$
\left|f_{n_k}(x)-f(x)\right|<\delta\frac{\varepsilon}{2}
$$
for all $x\in B_X$ such that $(x,y)\in K_\delta$ for some $y\in B_X$. Therefore, for all $k\geq k_0$, we have that 
$$
\frac{|(f_{n_k}-f)(x)-(f_{n_k}-f)(y)|}{\|x-y\|}<\varepsilon
$$
if $x,y\in B_X$ and $\|x-y\|\geq\delta$. It follows that $L(f_{n_k}-f)\leq\varepsilon$ for all $k\geq k_0$. By the arbitrariness of $\varepsilon$, we deduce that $\lim_{k\to\infty} L(f_{n_k}-f)=0$. Since $K$ is closed, it follows that $f\in K$. This proves that $K$ is compact.
\end{proof}

\begin{theorem}\label{teo-4-4}
If $X$ is a finite dimensional complex Banach space, then $f\mapsto f^t$ is an isometric isomorphism from $\Hl(B_X,Y)$ onto $\L((Y^*,bw^*);\Hl(B_X))$.
\end{theorem}

\begin{proof}
Let $f\in\Hl(B_X,Y)$. We claim that $f^t\in\L((Y^*,bw^*);\Hl(B_X))$. Indeed, note first that 
$$
\frac{\|f(x)-f(y)\|}{\|x-y\|}
=\sup_{y^*\in B_{Y^*}}\frac{|y^*(f(x)-f(y))|}{\|x-y\|}
=\sup_{y^*\in B_{Y^*}}\frac{|f^t(y^*)(x)-f^t(y^*)(y)|}{\|x-y\|}
$$
for every $x,y\in X$ with $x\neq y$. For each $y^*\in B_{Y^*}$, it follows that $f^t(y^*)\in\Hl(B_X)$ with $L(f^t(y^*))\leq L(f)$. Hence $f^t(B_{Y^*})$ is a bounded subset of $\Hl(B_X)$, and since  
$$
\lim_{\|x-y\|\to 0}\sup_{y^*\in B_{Y^*}}\frac{|f^t(y^*)(x)-f^t(y^*)(y)|}{\|x-y\|}=\lim_{\|x-y\|\to 0}\frac{\|f(x)-f(y)\|}{\|x-y\|}=0,
$$
Lemma \ref{compact-little} assures that $f^t(B_{Y^*})$ is relatively compact in $\Hl(B_X)$, that is, $f^t\in\K(Y^*,\Hl(B_X))$. Moreover, $f^t$ belongs to $\L((Y^*,w^*);(\HL(B_X),w^*))$ by Corollary \ref{teo-4-1}. Hence 
$$
\Lambda_X\circ f^t\in\K(Y^*,\G(B_X)^*)\cap\L((Y^*,w^*);(\G(B_X)^*,w^*)).
$$
Then there is an $S\in\L(\G(B_X),Y)$ for which $S^*=\Lambda_X\circ f^t$. 
Thus $S^*\in\K(Y^*,\G(B_X)^*)$ and, by the Schauder Theorem, $S\in\K(\G(B_X),Y)$. Then $f^t=\Lambda_X^{-1}\circ S^*\in\L((Y^*,bw^*);\HL(B_X))$. 
Since $f^t(Y^*)\subseteq\Hl(B_X)$, our claim is proved.

By Proposition \ref{Proposition: transpose properties}, the mapping of the statement is a linear isometry. To prove that it is onto, let $T\in\L((Y^*,bw^*);\Hl(B_X))$. By Proposition \ref{cor-4-1}, there is a $f\in\HL^\K(B_X,Y)$ such that $T=f^t$. We now show $f\in\Hl(B_X,Y)$. Since $T\in\K(Y^*,\HL(B_X))$ by Theorem \ref{teo-4-2} and $T(Y^*)\subseteq\Hl(B_X)$, we deduce that $T\in\K(Y^*,\Hl(B_X))$. Then Lemma \ref{compact-little} yields
\[\begin{split}
\lim_{\|x-y\|\to 0}\frac{\|f(x)-f(y)\|}{\|x-y\|}
&=\lim_{\|x-y\|\to 0}\sup_{y^*\in B_{Y^*}}\frac{|y^*(f(x)-f(y))|}{\|x-y\|}\\
&=\lim_{\|x-y\|\to 0}\sup_{y^*\in B_{Y^*}}\frac{|T(y^*)(x)-T(y^*)(y)\|}{\|x-y\|}=0,
\end{split}\]
and thus $f\in\Hl(B_X,Y)$.  
\end{proof}

We finish this paper with other elementary subclasses of holomorphic Lipschitz maps.

\begin{definition}
A map $f \in \HL(B_X,Y)$ is said to have finite-rank if $\lin(\Im_{L}(f))$ is a finite dimensional subspace of $Y$. The dimension of this space will be denoted by $\HL\text{-} \mathrm{rank}(f)$, and the set of all such maps by $\HL^{\F}(B_X,Y)$.

A map $f\in\HL(B_X,Y)$ is said to be approximable if it is the limit in the Lipschitz norm $L(\cdot)$ of a sequence of finite-rank holomorphic Lipschitz maps of $\HL(B_X,Y)$.

We denote by $\HL^{\F}(B_X,Y)$ and $\HL^{\overline{\F}}(B_X,Y)$ the spaces of finite-rank holomorphic Lipschitz maps and approximable holomorphic Lipschitz maps $f$ from $B_X$ into $Y$ such that $f(0)=0$, respectively.
\end{definition}

Applying both linearization and transposition, we describe finite-rank holomorphic Lipschitz maps as follows.

\begin{theorem}\label{Theorem: linearization finite-dimensional}
Let $f\in\HL(B_X,Y)$. The following assertions are equivalent:
\begin{enumerate}
\item $f\in\HL^{\F}(B_X,Y)$.
\item $T_f\in\F(\G(B_X),Y)$. 
\item $f^t\in\F(Y^*,\HL(B_X))$. 
\end{enumerate}
In such a case, we have 
$$
\HL \text{-} \mathrm{rank}(f) = \mathrm{rank}(T_f)=\mathrm{rank}(f^t).
$$
Moreover, the maps $f \mapsto T_f$ and $f \mapsto f^t$ are isometric isomorphisms from $(\HL^{\F}(B_X,Y),L(\cdot))$ onto $(\F(\G(B_X),Y),\left\|\cdot\right\|)$ and onto $(\F(Y^*,\HL(B_X)),\left\|\cdot\right\|)$, respectively.
\end{theorem}

\begin{proof}
First note that  $\lin(\Im_{L}(f))=\lin(f(B_X))$.

$(i)\Rightarrow (ii)$: If $f\in\HL^{\F}(B_X,Y)$, then $\lin(\Im_{L}(f))$ is finite dimensional. Applying \cite[Proposition 2.3 (b)-(c)]{Aro24}, we deduce that
\[\begin{split}
T_f(\G(B_X)) &= T_f(\overline{\lin}(\delta_X(B_X))) \subseteq \overline{T_f(\lin(\delta_X(B_X)))}\\
&= \overline{\lin}(T_f(\delta_X(B_X))) = \overline{\lin}(f(B_X))=\lin(f(B_X)),
\end{split}\]
and thus $T_f\in\F(\G(B_X),Y)$ with $\mathrm{rank}(T_f)\leq\HL \text{-} \mathrm{rank}(f)$.

$(ii)\Rightarrow (i)$: If $T_f\in\F(\G(B_X),Y)$, then
\[\begin{split}
\lin(f(B_X)) &= \lin(T_f(\delta_X(B_X))) = T_f(\lin(\delta_X(B_X)))\\
&\subseteq T_f(\overline{\lin}(\delta_X(B_X))) = T_f(\G(B_X)),
\end{split}\]
and so $f\in\HL^{\F}(B_X,Y)$ with $\HL\text{-}\mathrm{rank}(f)\leq\mathrm{rank}(T_f)$.

$(ii)\Leftrightarrow (iii)$ and $\mathrm{rank}(T_f)=\mathrm{rank}(f^t)$ follow from Proposition \ref{Proposition: transpose properties} and the fact that the operator ideal $[\F,\left\|\cdot\right\|]$ is completely symmetric \cite[4.4.7]{Pie80}.
\end{proof}

An application of Theorems \ref{Theorem: linearization}, \ref{Theorem: dual ideal} and \ref{Theorem: linearization finite-dimensional} yields the following.

\begin{corollary}
We have 
$$
[\HL^{\F},L(\cdot)]
=[\F\circ\HL,\left\|\cdot\right\|_{\F\circ\HL}]
=[\F^{\mathrm{dual}}\circ \HL,\left\|\cdot\right\|_{\F^{\mathrm{dual}}\circ \HL}]
=[\F^{\HL\text{-}\mathrm{dual}},\left\|\cdot\right\|_{\F^{\HL\text{-}\mathrm{dual}}}].
$$ $\hfill\qed$

\end{corollary}

Theorem \ref{Theorem: linearization finite-dimensional} shows the connection of an approximable holomorphic Lipschitz map with its linearization:

\begin{corollary}\label{prop2.1.new}
Let $f\in\HL(B_X,Y)$. Then $f\in\HL^{\overline{\F}}(B_X,Y)$ if and only if $T_f\in\overline{\F}(\G(B_X),Y)$. Moreover, the map $f\mapsto T_f$ is an isometric isomorphism from $\HL^{\overline{\F}}(B_X,Y)$ onto $\overline{\F}(\G(B_X),Y)$. $\hfill\qed$
\end{corollary}

Note that $\HL^{\F}(B_X,Y)$ is a linear subspace of $\HL^{\K}(B_X,Y)$. In fact, we have a little more:

\begin{corollary}
Every approximable holomorphic Lipschitz map from $B_X$ to $Y$ is compact holomorphic Lipschitz.
\end{corollary}

\begin{proof}
Let $f\in\HL^{\overline{\F}}(B_X,Y)$. Then there is a sequence $(f_n)$ in $\HL^{\F}(B_X,Y)$ such that $L(f_n-f)\to 0$ as $n\to\infty$. Since $T_{f_n}\in\F(\G(B_X),Y)$ by Theorem \ref{Theorem: linearization finite-dimensional} and $\left\|T_{f_n}-T_f\right\|=\left\|T_{f_n-f}\right\|=L(f_n-f)$ for all $n\in\mathbb{N}$, we deduce that $T_f\in\K(\G(B_X),Y)$ since $[\K,\|\cdot\|]$ is a closed ideal of $[\L,\|\cdot\|]$, and so $f\in\HL^{\K}(B_X,Y)$ by Theorem \ref{teo-4-2}.
\end{proof}

Let us recall that a Banach space $X$ has the approximation property if given a compact set $K\subseteq X$ and $\varepsilon>0$, there is an operator $T\in\F(X,X)$ such that $\left\|T(x)-x\right\|<\varepsilon$ for every $x\in K$. We refer to \cite[Section 3]{Aro24} for some results on the approximation property of holomorphic Lipschitz maps.

Grothendieck \cite{g} proved that a dual Banach space $X^*$ has the approximation property if and only if given a Banach space $Y$, an operator $S\in\K(X,Y)$ and $\varepsilon>0$, there is an operator $T\in\F(X,Y)$ such that $\left\|T-S\right\|<\varepsilon$. Since $\HL(B_X)$ is a dual space, combining the Grothendieck's result with Theorems \ref{teo-4-2} and \ref{Theorem: linearization finite-dimensional}
, we propose the next characterization.

\begin{corollary}
$\HL(B_X)$ has the approximation property if and only if, for each complex Banach space $Y$, $\varepsilon>0$ and $f\in\HL^{\K}(B_X,Y)$, there exists $g\in\HL^{\F}(B_X,Y)$ such that $L(f-g)<\varepsilon$.
$\hfill\qed$
\end{corollary}









\end{document}